\documentclass[12pt]{amsart}
\usepackage{amscd,epsfig,fullpage,url}
\usepackage[colorlinks=true, pdfstartview=FitV, linkcolor=blue, citecolor=blue, urlcolor=blue]{hyperref}
\usepackage{amssymb}
\usepackage{comment}
\usepackage{graphicx}
\usepackage{tikz-cd} 

\usepackage[normalem]{ulem}

\title[Bracket width of the Lie algebra of vector fields on a smooth affine curve]{Bracket width of the Lie algebra of vector fields on a smooth affine curve}
\author{Ievgen Makedonskyi  and  Andriy Regeta}

\address{\noindent Institut f\"{u}r Mathematik, Friedrich-Schiller-Universit\"{a}t Jena, \newline
\indent  Jena 07737, Germany}
\email{makedonskyi.e@gmail.com, andriyregeta@gmail.com}

\thanks{  }

\newtheorem{claim}{Claim}

\newtheorem{corollary}{Corollary}

\newtheorem{definition and proposition}{Definition and proposition}
\newtheorem{conjecture}{Conjecture}

\newtheorem*{thA}{Theorem A}
\newtheorem*{thB}{Theorem B}
\theoremstyle{definition}
\newtheorem{definition}{Definition}

\newcommand{\Der}{\operatorname{Der}}

\newcommand{\dx}{\frac{\partial}{\partial x}}
\newcommand{\dy}{\frac{\partial}{\partial y}}
\newcommand{\dz}{\frac{\partial}{\partial z}}

\DeclareMathOperator{\Ve}{Vec}

\def\Dz(z-1){\mathrm{D}_{z(z-1)}}

\def \itt #1,#2:{\medskip\item[$\bullet$] 
     page\ \ignorespaces#1, line\ \ignorespaces#2:\ \ignorespaces}

\begin{document}

\begin{abstract}  
We prove that the bracket width of the simple Lie algebra of vector fields $\Ve(C)$ of a smooth irreducible affine curve $C$ with a trivial tangent sheaf is at most three.  In addition, if $C$ is a plane curve, the bracket width of $\Ve(C)$ is at most two and if moreover $C$ has a unique place at infinity, the bracket width of $\Ve(C)$ is exactly two.
 We also show that
 in case $C$ is rational, the width of  $\Ve(C)$ equals one. 
\end{abstract}

\maketitle

%%%%%%%%%%%%%%%%%%%%%%%%%%%%%%%%%%%%%%%%%

\section{Introduction}
 
 Given a Lie  algebra $L$ over an infinite field $\Bbbk$, we define its bracket width  as the supremum of lengths $\ell (x)$, where $x$ is runs over the derived algebra $[L,L]$ and $\ell (x)$ is defined as the smallest number $n$ of Lie brackets $[y_i,z_i]$ needed to 
represent $x$ in the form 
$$
\sum_{i=1}^n[y_i,z_i]. 
$$
The bracket width applies in studying different aspects of Lie algebras,  see \cite{Rom}. In particular,  in \cite{Rom} the author provides many examples of Lie algebras with the bracket width strictly bigger than one. However,
the first example of a simple Lie algebra with the bracket width strictly bigger than one was found only very recently in  \cite[Theorem A]{DKR21}
among 
 Lie algebras of vector fields of smooth affine curves which are simple by \cite{Jor} and \cite[Proposition~1]{Sie96}. In the current note we provide an upper bound on the bracket width of a Lie algebra  of vector fields
  on an irreducible smooth affine curve $C$ with  certain properties. 
 %In particular, \cite[Example 2]{DKR21} shows that the bracket width of the Lie algebra $\Ve(C)$ is bigger or equal than two, where $C$ is a  smooth plane affine curve defined
%by the equation $y^2 = h(x)$ for some separable polynomial $h(x)$ of degree $2g + 1$. 
%Hence, it is of interest to provide an upper bound for the bracket width of such Lie algebras of vector fields of smooth affine curves.
Our first main result is the following statement which partially answers \cite[Question 2]{DKR21}.

\begin{thA}\label{theoremA}
Let $C$ be an irreducible smooth affine curve with  trivial tangent sheaf. Then the bracket width of the Lie algebra $\mathrm{Vec}(C)$ is smaller than or equal to three.  In addition, if $C$ is a plane curve, the bracket width of $\mathrm{Vec}(C)$ is smaller than or equal to two. 
\end{thA}

The upper bound for the bracket width given in Theorem A is, in particular, of interest since it allows us to compute the bracket width of $\Ve(C)$ for a certain family of smooth plane  affine curves.

\begin{corollary}\label{brackethyperbolic}
Let $C$ be an irreducible non-rational smooth plane affine curve with a unique place at infinity.
%Assume that $C$ is not rational. 
Then the width
of the simple Lie algebra $\Ve(C)$  equals two.
\end{corollary}

There are many examples of affine curves with only one place at infinity (see Definition \ref{defoneplace}), and they were
studied in many different contexts, see, e.g., a  paper of Koll\'ar \cite{Kol} and the references
therein. A simple class of examples is given by affine hyperelliptic plane curves $C \subset \mathbb{A}^2$
defined by equations  $y^2 = h(x)$, where $h(x)$ is a  monic polynomial
of odd degree strictly greater than one which has only simple roots (\cite[Example 2]{DKR21}).

We believe that the assumption on the curve $C$ in Corollary \ref{brackethyperbolic} can be  lightened and we have the following conjecture.
\begin{conjecture}
Assume $C$ is a non-rational affine smooth plane curve. Then the bracket width of $\Ve(C)$ is exactly two.
\end{conjecture}

Assume $f$ is a regular function on $C$. We define a principal open subset $C_f \subset C$ as 
\[
\left\{ x \in C \mid f(x) \neq 0 \right\} \subset C.
\]
Note that $C_f$ is a smooth affine curve itself.

\begin{thB} 
Let $C$ be an irreducible smooth affine curve and $C_f$ be its  principal open subset. 
%with a  trivial tangent sheaf. 
Then the bracket width of  $\mathrm{Vec}(C_f)$ is
smaller than or equal to
 the bracket width of $\mathrm{Vec}(C)$.
\end{thB}
We do not know an example of a smooth affine curve $C$ with a principal open subset $U \subset C$ such that the width of $\Ve(U)$ is strictly smaller than $\Ve(C)$.

As a consequence of Theorem B we have the following statement that disproves \cite[Conjecture 1]{DKR21}.

\begin{corollary}\label{bracketrationalcurve}
If $\Bbbk$ is algebraically closed,
the bracket width of the Lie algebra $\mathrm{Vec}(C)$ of a rational smooth affine curve $C$ is one.
\end{corollary}

\medskip

%\section{Vector fields on smooth affine curves} \label{sec:curves} 

%In this section, we consider Lie algebras of algebraic vector fields on certain smooth affine curves. We begin with the case of rational curves. Recall that every such curve $C$ is isomorphic to a principal Zariski open subset of the affine line $\mathbb{A}^1=\mathrm{Spec}(k[x])$. 

%In contrast with the case of the $1$-dimensional torus $\mathbb{T}^1\simeq \mathbb{A}^1\setminus \{0\}$, for which the vector field $x^{-1}\partial_x$ can be written as the single Lie bracket $[x^{m}\partial_x,-\frac{1}{2m}x^{-m}\partial_x]$ for any $m\neq 0$, we expect that for every  $n\geq 2$, a vector field on $\mathbb{A}^1\setminus \{p_1,\dots, p_n\}$ with single pole at each of the points $p_1,\dots p_n$ cannot be written as the Lie bracket of two elements of $\Ve(\mathbb{A}^1\setminus \{p_1,\dots, p_n\})$. This motivates the following:
 
\section{Proof of Theorem A}

\begin{proof}[Proof of Theorem A]
Denote by $\mathcal{O}(C)$ the ring of regular functions on $C$.
Since by hypothesis the tangent sheaf of $C$ is trivial, we have $\Ve(C) = \mathcal{O}(C) \cdot \tau$ for a certain
nowhere vanishing global vector field $\tau \in \Ve(C)$, unique up to multiplication by a nonzero
constant.
It is well-known that every smooth affine variety of dimension $d$ can be embedded
into $\mathbb{A}^{2d+1}$ and that the bound $2d + 1$ is optimal (\cite[Corollary 1]{Sr91}). In particular, a smooth affine curve can be embedded
into $\mathbb{A}^{3}$ and this bound is sharp. 
Hence, $\mathcal{O}(C) \simeq \Bbbk[x,y,z]/I$, where $I \subset \Bbbk[x,y,z]$ is some ideal and we have the natural surjections
\[\pi\colon\Bbbk[x,y,z] \twoheadrightarrow \Bbbk[x,y,z]/I\]
and 
\[\pi_*\colon \left\{ \nu \in \Der\Bbbk[x,y,z] = \Ve(\mathbb{A}^3) \mid \nu (I)\subset I \right\} \twoheadrightarrow \Ve(C).\] 
Note that $\left\{ \nu \in \Der\Bbbk[x,y,z] \mid \nu (I)\subset I \right\} \subset \Der\Bbbk[x,y,z]$ is a Lie subalgebra and $\pi_*$ is a homomorphism of Lie algebras.
Then
$\tau$ is the image of a derivation $\tilde \tau =\tilde P\dx + \tilde Q \dy + \tilde R\dz$ that preserves the ideal $I$.  Further, define $P,Q,R \in \mathcal{O}(C)$, $P=\pi(\tilde{P})$, $Q=\pi(\tilde{Q})$, $R=\pi(\tilde{R})$.
 %We claim that $P,Q,R$ are coprime. Indeed, if $P,Q,R \in \mathcal{O}(C)$ are not coprime, $\tau$ would preserve a non-trivial ideal $J$ of $\mathcal{O}(C)$ generated by $P,Q$ and $R$ which contradicts the simplicity of $\Ve(C)$ as
%\begin{equation}\label{non-simplicity}
%[\Ve(C), J\tau] = [\mathcal{O}(C)\tau,J\tau] \subset J\tau.
%\end{equation}

%Assume first that $C$ is not a plane curve, i.e., $C$ can not be realized as a closed subset of $\mathbb{A}^2$. This implies that the regular functions $y,z \in \mathcal{O}(\mathbb{A}^3) = \Bbbk[x,y,z]$ induce the non-zero regular functions on $C$.
\begin{claim}\label{claim11}
\[ \left\{ [\tilde\tau,\tilde f\tilde\tau]  
 +  [y\tilde\tau,\tilde g\tilde\tau] +  [z\tilde\tau,\tilde h\tilde\tau] \mid \tilde f,\tilde g, \tilde h \in \Bbbk[x,y,z]  \right\} = (\tilde P,\tilde Q, \tilde R) \tau,
 \]
 where $(\tilde P,\tilde Q, \tilde R)$ denotes the ideal of  $\Bbbk[x,y,z]$ generated by $\tilde P,\tilde Q$ and $\tilde R$.
\end{claim}

%Indeed, for some polynomials $\tilde f, \tilde g, \tilde h \in \Bbbk[x,y,z]$ we have:
Indeed,
\begin{eqnarray}\nonumber
& [\tilde\tau,\tilde f\tilde\tau]  
 +  [y\tilde\tau,\tilde g\tilde\tau] +  [z\tilde\tau,\tilde h\tilde\tau] = (\tilde P\tilde f'_x + \tilde Q \tilde f'_y + \tilde R \tilde f'_z + \\ & y(\tilde P \tilde g'_x + \tilde Q \tilde g'_y + \tilde R \tilde g'_z) - \;  
\tilde g \tilde Q +  z(\tilde P \tilde h'_x + \tilde Q \tilde h'_y + \tilde R \tilde h'_z) - \tilde h \tilde R)\tilde\tau 
 \label{sumofthreebrackets}
 \end{eqnarray}
which equals 
\begin{multline}
(\tilde P(\tilde f'_x + y\tilde g'_x + z\tilde h'_x) + \tilde Q (\tilde f'_y + y \tilde g'_y + z \tilde h'_y -\tilde g) + \tilde R(\tilde f'_z + y\tilde g'_z +z\tilde h'_z -\tilde h))\tilde\tau 
 \\ =
 (\tilde P(\tilde f + y\tilde g + z\tilde h)'_x + \tilde Q ((\tilde f + y \tilde g + z \tilde h)'_y -2 \tilde g) +
\tilde R((\tilde f + y \tilde g + z\tilde h)'_z -2 \tilde h))\tilde\tau .
\label{equality}\end{multline}
%Note that 
%\[ 
%\bigl\{ f + yg + zh  \bigl\} =
%\bigl\{ (f + y g + zh)'_y -2g 
%\bigl\}=
%\bigl\{ (f + y g + zh)'_z -2h \bigl\}.
%\]
Define $\tilde r = \tilde f + y \tilde g + z \tilde h \in \mathcal{O}(C)$.
Now, the expression \eqref{equality} can be written as
\[
( \tilde P \tilde r'_x +  \tilde Q ( \tilde r'_y -2 \tilde g) +  \tilde R( \tilde r'_z -2 \tilde h))\tilde\tau .
\]
For any $F,G,H \in \Bbbk[x,y,z]$ there exist $ \tilde r, \tilde g, \tilde h $ such that $F =  \tilde r'_x$, $G =  \tilde r'_y - 2 \tilde g$ and $H= \tilde r'_z - 2 \tilde h$.
Hence, 
\begin{eqnarray*}\nonumber
  & \left\{ ( \tilde P \tilde r'_x +   \tilde Q ( \tilde r'_y -2 \tilde g) +  \tilde R( \tilde r'_z -2 \tilde h))\tilde\tau \mid \tilde r, \tilde g, \tilde h \in \Bbbk[x,y,z] \right\} =  \\ & 
\left\{( \tilde PF +  \tilde QG+ \tilde RH)\tilde\tau \mid F,G,H \in \Bbbk[x,y,z] \right\} = (\tilde P,\tilde Q, \tilde R) \tilde\tau.
 \end{eqnarray*}
This proves Claim \ref{claim11}.

We claim now that any vector field from $\Ve(C)$ can be written as the sum
\begin{equation}
    [\tau,f\tau] +  [\pi(y)\tau,g\tau] +  [\pi(z)\tau,h\tau]
\end{equation}
for some regular functions $f,g,h \in \mathcal{O}(C)$.
%This statement follows from the next claim.
Indeed,
 $(\tilde P,\tilde Q, \tilde R)$ is preserved by $\tilde \tau$, hence $\pi((\tilde P,\tilde Q, \tilde R))$ is preserved by $\tau$.
 Whence $\pi((\tilde P,\tilde Q, \tilde R))\tau$ is the ideal in $\Ve(C)$.
 Therefore, because of simplicity of $\Ve(C)$ we have
\[ \pi_*((\tilde P,\tilde Q, \tilde R)\tau) =\Ve(C) \text{ or equivalently }  \pi((\tilde P,\tilde Q, \tilde R)) =\mathcal{O}(C).\]
Finally, by Claim \ref{claim11} we have
\[\left\{ \pi_*([\tilde\tau,\tilde f\tilde\tau])  
 +  \pi_*([y\tilde\tau,\tilde g\tilde\tau]) +  \pi_*([z\tilde\tau,\tilde h\tilde\tau]) \mid \tilde f,\tilde g, \tilde h \in \Bbbk[x,y,z]  \right\} = \pi_*((\tilde P,\tilde Q, \tilde R) \tau) = \Ve(C).
 \]
Thus, the first statement of the theorem follows as $\pi_*$ is a homomorphism of Lie algebras. 

If $C$ is a plane curve, then, similarly as above, $\Ve(C) = \mathcal{O}(C) \cdot \tau$,
$\mathcal{O}(C) \simeq \Bbbk[x,y]/I$, where $I \subset \Bbbk[x,y]$ is an ideal and we have the natural surjections
\[\pi\colon\Bbbk[x,y] \twoheadrightarrow \Bbbk[x,y]/I\]
and 
\[\pi_*\colon \left\{ \nu \in \Der\Bbbk[x,y]=\Ve(\mathbb{A}^2) \mid \nu(I) \subset I \right\} \twoheadrightarrow \Ve(C).\] Then
$\tau$ is the image of a derivation $\tilde \tau =\tilde P\dx + \tilde Q \dy$ that preserves the ideal $I$. Further,
define $P,Q \in \mathcal{O}(C)$, $P=\pi(\tilde{P})$, $Q=\pi(\tilde{Q})$.
 %We claim that $P,Q,R$ are coprime. Indeed, if $P,Q,R \in \mathcal{O}(C)$ are not coprime, $\tau$ would preserve a non-trivial ideal $J$ of $\mathcal{O}(C)$ generated by $P,Q$ and $R$ which contradicts the simplicity of $\Ve(C)$ as
%\begin{equation}\label{non-simplicity}
%[\Ve(C), J\tau] = [\mathcal{O}(C)\tau,J\tau] \subset J\tau.
%\end{equation}
%Assume first that $C$ is not a plane curve, i.e., $C$ can not be realized as a closed subset of $\mathbb{A}^2$. This implies that the regular functions $y,z \in \mathcal{O}(\mathbb{A}^3) = \Bbbk[x,y,z]$ induce the non-zero regular functions on $C$.
The second statement of the theorem follows if we prove that
any vector field of $\Ve(C)$ can be written as the sum
\begin{equation}\label{formulaplanecurve}
    [\tau,f\tau] +  [\pi(y)\tau,g\tau]
\end{equation}
for some regular functions $f,g \in \mathcal{O}(C)$.
Similarly as above \eqref{formulaplanecurve} follows from the next equality
\[ \left\{ [\tilde\tau,\tilde f\tilde\tau]  
 +  [y\tilde\tau,\tilde g\tilde\tau] \mid \tilde f,\tilde g, \in \Bbbk[x,y]  \right\} = (\tilde P,\tilde Q) \tilde\tau
 \]
 which is proved analogously as Claim \ref{claim11}.
% for some $f,g \in \mathcal{O}(C)$. Indeed, define $r = f+yg \in \mathcal{O}(C)$ and note that 
% \[
% \bigl\{
% P(f'_x + yg'_x) + Q(f'_y + yg'_y - g) \mid f,g \in \mathcal{O}(C)
% \bigl\} = \bigl\{ Pr'_x + Q(r'_y - 2g) \mid r,g \in \mathcal{O}(C)
% \bigl\}.
% \]
% Now, 
% for any $F,G \in \mathcal{O}(C)$ there exist $r,g \in  \mathcal{O}(C)$ such that $F = r'_x$ and $g = r'_y - 2g$.
% Consequently, 
% \[
% \{ (Pr'_x + Q (r'_y -2g))\tau \mid r,g \in \mathcal{O}(C)  \} = \{ PF + QG \mid F,G \in \mathcal{O}(C) \} = J \tau,
% \]
% where $J \subset \mathcal{O}(C)$ is the ideal generated by $P$ and $Q$.
% Similarly as in \eqref{non-simplicity} $\Ve(C)$ is not simple if $J \neq \mathcal{O}(C)$.
% Therefore, $J = \mathcal{O}(C)$ and the proof follows.
\end{proof}

 \begin{definition}\label{defoneplace} An irreducible smooth affine curve $C$ is said to have \emph{a unique place at infinity} if it is equal to the complement of a single closed point in a smooth projective curve $\bar{C}$.
 \end{definition}

\begin{proof}[Proof of Corollary \ref{brackethyperbolic}]
By \cite[Theorem A]{DKR21} the bracket width of $\Ve(C)$ is strictly greater than one. 
Since $C$ is a smooth irreducible plane affine curve, its tangent sheaf is trivial.
Now by Theorem A the width of $\Ve(C)$ is less or equal than two which proves the claim.
\end{proof}

%Let us recall that  we  an affine curve $C$ without a point is still affine curve. This easy observation follows from the fact that a principal open subset of $C$ is an affine curve itself and principal open subsets form a basis of the Zariski topology on $C$.

%\begin{corollary}\label{removetoobtainplanecurve}
%Let $C$ be an irreducible smooth  affine curve.
% Then there exists a finite subset $\Lambda \subset C$ such that for any finite subset $\tilde{\Lambda} \subset (C \setminus \Lambda)$ the width of $\Ve(C \setminus (\Lambda \cup \tilde{\Lambda}))$ is at most two.
%\end{corollary}
%\begin{proof}
%By \cite[Lemma 3.12]{Jel} there exists a finite subset $\Lambda \subset C$ such that %for any finite subset $\tilde{\Lambda}  \subset (C \setminus \Lambda)$ 
%the  curve is $C \setminus \Lambda$ is a plane affine curve.   Hence, the statement follows from Theorem A.
%\end{proof}

\section{Proof of Theorem B}

\begin{proof}[Proof of Theorem B]
%It is well known that 
%if $X$ is irreducible affine variety, then $\Ve(X)$ is a torsion-free $\mathcal{O}(X)$-module of rank equal
%to $\dim X$ (see \cite[Lemma 4]{Kr17} for an easy argument).
Denote by $\mathcal{K}(C)$ the field of fractions of $\mathcal{O}(C)$. Then 
\begin{equation}\label{tensorproduct}
\mathcal{K}(C) \otimes_{\mathcal{O}(C)} \Der(\mathcal{O}(C)) \simeq
\Der(\mathcal{K}(C))
\end{equation}
and the latter is known to be a free $\mathcal{K}(C)$-module of rank equal to   $\dim C = 1$.
%Hence,
%$\Ve(C) = \Der(\mathcal{O}(C))$ is a finitely generated module over $\mathcal{O}(C)$. 
By \eqref{tensorproduct} $\Ve(C) =\Der\mathcal{O}(C)\subset \Der\mathcal{K}(C) = \mathcal{K}(C)\tau$, where
 $\tau \in  \Ve(C)$ is a  global vector field.
Hence,  $\Ve(C)=\mathcal{R}(C) \cdot \tau$ for a certain
% finite 
$\mathcal{O}(C)$-module $\mathcal{R}(C) \subset \mathcal{K}(C)$.
%Moreover, $\mathcal{O}(C)$ is a Dedekind domain and in a Dedekind domain every ideal is either principal or generated by two elements.
%Hence, the maximal ideal $I$ of 
%$\mathcal{O}(C)$ that corresponds to the point $p \in C$ is generated by some $f,g \in \mathcal{O}(C)$.
As a consequence, $\Ve(C_{f})=\mathcal{R}(C) [f^{-1}]\tau$.  
 Now,
 for $a,b \in \mathcal{R}(C)$, $\frac{a}{f^k} \tau, \frac{b}{f^k} \tau \in \Ve(C_{f}) = \mathcal{R}(C)[f^{-1}] \tau$ and we have:
\begin{equation}\label{devidedcommutator}
  \left[\frac{a}{f^k} \tau,\frac{b}{f^k}\tau\right]=\left(   \frac{a}{f^k}\tau(\frac{b}{f^k}) - \frac{b}{f^k}\tau(\frac{a}{f^k}) \right)\tau=   %\frac{a}{f_1^{2k}} \tau(b) + \frac{ab}{f_1^k} \tau(f_1) - \frac{b}{f_1^{2k}} \tau(a) - \frac{ab}{f_1^k} \tau(f_1) = 
  \frac{1}{f^{2k}}\left[a,b \right]\tau. 
\end{equation}
Further,
for any $h \in \mathcal{R}(C)[f^{-1}]$ there exists $g \in \mathcal{R}(C)$ such that 
\[h=\frac{g}{f^{2k}}\]
for some $k \in \mathbb{N}$.
Assume that
\[g\tau=[a_1\tau,b_1\tau]+\dots+[a_n\tau,b_n\tau].\]
Then using  \eqref{devidedcommutator} we have
\[h\tau = \frac{g}{f^{2k}} \tau =\left[\frac{a_1}{f^k}\tau,\frac{b_1}{f^k}\tau\right]+\dots+\left[\frac{a_n}{f^k}\tau,\frac{b_n}{f^k}\tau\right].\]
This completes the proof.
\end{proof}

\begin{proof}[Proof of Corollary \ref{bracketrationalcurve}]
Let us recall that every rational smooth affine curve $C$ is isomorphic to
$\mathbb{A}^1 \setminus \Lambda$, where $\Lambda$ is a finite set of $r \ge 0$ points.
In particular, it admits a closed embedding into $\mathbb{A}^2$. Indeed, $C$ can also
be seen as the complement in $\mathbb{A}^1$ of a finite number (possibly zero) of points
and we can consider the closed embedding $\mu \colon C \to \mathbb{A}^2$ given by $x \mapsto (x, \frac{1}{f(x)} )$,
where $f \in \Bbbk[x]$ is a polynomial whose roots are exactly the removed points.
Note that the image of $\mu$ is the curve of $\mathbb{A}^2$ defined by the equation $f(x)y
= 1$.

 By \cite[Proposition 1]{DKR21} the bracket width of $\mathrm{Vec}(\mathbb{A}^1)$ is one. By Theorem B the bracket width of $\Ve(\mathbb{A}^1 \setminus \Lambda)$ 
 is smaller than or equal to 
  the bracket width of  $\Ve(\mathbb{A}^1)$ which is one. The proof follows.
\end{proof}

%\begin{remark}
%As it is mentioned after Theorem B, we do not know the example of a curve $C$ with a principal open subset $U \subset C$ such that the width of $\Ve(U)$ is smaller than $\Ve(C)$.
%But such an example exists if there
%exists a smooth affine curve $C$ with the width of $\Ve(C)$ equal to $3$. Indeed, as it is already mentioned in the proof of Corollary \ref{removetoobtainplanecurve}, there exists 
% a finite subset $\Lambda \subset C$ such that for any finite subset $\tilde{\Lambda}  \subset (C \setminus \Lambda)$ the  curve  $C \setminus (\Lambda \cup \tilde{\Lambda})$ is a plane curve. One can always choose a principal open subset $U=C_f$ of $C$ such that $f$ vanishes at all points of $\Lambda$. This implies that the width of $\Ve(U)$ is smaller or equal than $2$.
%\end{remark}

%%%%%%%%%%%%%%%%%%%%%%%%%%%%%%%

\medskip

\noindent{\it Acknowledgements}. We thank Adrien Dubouloz and Boris Kunyavskii for useful discussions.

%\vskip1cm

\end{document}